\documentclass[leqno,12pt]{amsart}
\usepackage{stmaryrd} 
\usepackage{amssymb}
\theoremstyle{plain} 
\newtheorem{theorem}{\indent\sc Theorem}[section]
\newtheorem{lemma}[theorem]{\indent\sc Lemma}

\theoremstyle{definition} 

\newtheorem{remark}[theorem]{\indent\sc Remark}

\begin{document}

\title[  A generalized Cahn-Hilliard equation]{ Global existence of a generalized Cahn-Hilliard equation with  biological applications} 


\author[N. Duan]{Ning Duan}
\author[X. Zhao]{Xiaopeng Zhao
} 
%

\address{N. Duan\endgraf
School of Science\endgraf
Jiangnan University\endgraf
Wuxi 214122,~~~
P. R. China
}
\email{dn@jiangnan.edu.cn}

\address{X. Zhao\endgraf
School of Science\endgraf
Jiangnan University\endgraf
Wuxi 214122,~~~
P. R. China\endgraf
and
\endgraf
Department of Mathematics\endgraf Southeast University\endgraf
Nanjing 210018,~~~P. R. China
}
\email{zhaoxiaopeng@jiangnan.edu.cn}

\begin{abstract}
In this paper,  on the basis of the Schauder type estimates and Campanato spaces, we prove the global existence of classical
solutions  for a generalized Cahn-Hilliard equation with biological applications.

\end{abstract}\maketitle
{\small\section{Introduction} \label{sect1}

In \cite{14}, Khain and Sander proposed a generalized Cahn-Hilliard equation
\begin{equation}
\label{1-00}\frac{\partial u}{\partial t}-\frac{\partial^2}{\partial x^2}\left[\ln(1-q)\frac{\partial^2u}{\partial x^2}+F'(u)\right]+\alpha u(u-1)=0.
\end{equation}
Equation (\ref{1-00}) is modelling cells which move, proliferate and interact via adhesion in wound healing and tumor growth. Here, $u$ is the local density of cells, $q$ is the adhesion parameter, $\alpha>0$ is the proliferation rate, $F$ is the local free energy. Moreover,
$$
q=1-\exp(-\frac J{k_BT}),
$$
where $J$ corresponds to the interatomic interaction, $k_B$ is the Boltzmann's constant and $T$ is the absolute temperature, assumed constant.
 Recently, for simplicity, Cherfils, Miranville and Zelik\cite{AM2} set all physical constants equal to $1$ and solved the problem in the higher space dimension (in two space dimensions, the equation models, e.g., the clustering of malignant brain tumor cells, see\cite{AM2,AM1}),
i.e., they studied asymptotic behavior the generalized Cahn-Hilliard equation
\begin{equation}
\label{1-0}\frac{\partial u}{\partial t}+\Delta^2u-\Delta f(u)+g(u)=0
\end{equation} endowed with Neumann boundary conditions. In addition, Zhao\cite{ZXP} studied the global solvability and dynamical behavior of solutions for the Cauchy problem of the  modified equation of (\ref{1-1}).

It is well-known that the principal part of many types of nonlinear diffusion equations which arise from mathematics and other branches of natural science (for example, physics, mechanics,material science, population ecology and so on) are  nonlinearity.
In the last three decades, more and more authors paid their attentions to the well-posedness of solutions for higher-order diffusion equations together with nonlinear principal parts, see for example \cite{Yin,Liu,Dai,GM} and so on.

In this paper, setting the nonlinear functions $f(u)=u^3$ and $g(u)=u^2$, we consider the following fourth order nonlinear parabolic equation
\begin{equation}
\label{1-1}
\partial_tu+D^2[a(u)D^2 u-u^3]+u^2=0,\quad(x,t)\in\Omega\times(0,T),
\end{equation}
where $\Omega=(0,1)$, $D=\frac{\partial}{\partial x}$, $a(u)$ is a nonlinear function. It is thus clear that the study on Eq.(\ref{1-1}) is  meaningful.
On the basis of physical consideration,
Eq.(\ref{1-1}) is supplemented by the following boundary condition
\begin{equation}
\label{1-2}
u(0,t)=u(1,t)=D^2u(0,t)=D^2u(1,t)=0,
\end{equation}
and the initial condition
\begin{equation}\label{1-3}
u(x,0)=u_0(x).
\end{equation}
\begin{remark}
In \cite{AM2}, the authors suppose that $f(s)=s^3-s$ and $g(s)=s^2-s$.  In order to simple the calculations, we delete the linear terms in $f(s)$ and $g(s)$. In fact, our main result also holds for the case $f(s)=s^3-s$ and $g(s)=s^2-s$.
\end{remark}

Our main purpose is to establish the global existence of classical solutions under much more general assumptions.
The main difficulties for treating the problem (\ref{1-1})-(\ref{1-3}) are caused by the nonlinearity of the principal part and the nonlinear term $u^2$. Due to the nonlinearity of the principal part, there are more difficulties in establishing
the global existence of classical solutions. Our method for investigating the regularity of solutions is based on uniform
Schauder type estimates for local in time solutions. Since the term $u^2$ is a polynomial of order $2$ on $u\in\mathbb{R}^1$, it is difficulty to deal with this term on the process of a prior estimates.  Employing the techniques in \cite{AM1}, we introduce the inverse operator $(-D^2)^{-1}$, which is a positive self-adjoint operator, to handle the nonlinear term $u^2$.
Our approach lies in the combination of the Schauder type estimates with some methods based on the
framework of Campanato spaces.

\begin{theorem}
\label{thm1.1}Suppose that
\begin{itemize}
\item $u_0(x)\in C^{4+\alpha}(\mathbb{R})$ for some $\alpha\in(0,1)$ and $\langle u_0\rangle=\int_{\Omega}u_0dx$ is bounded;
\item $a(u)$ if of class $C^{2+\alpha}(\mathbb{R})$ for some $\alpha\in(0,1)$,  $1<M_1\leq a(u)\leq M_2$ and $a'(u)u\geq0$, where $M_1$ and $M_2$ are two positive constants.
    \end{itemize} Then for any smooth initial value $u_0(x)$ with $u_0|_{x=0,1}=D^2u_0|_{x=0,1},
$  problem (\ref{1-1})-(\ref{1-3}) admits a unique classical solution.
\end{theorem}

\begin{remark}
\label{rem1.1}
Under the assumption of Theorem \ref{thm1.1}, Eq.(\ref{1-1}) can be rewritten, equivalently, as
\begin{equation}
\label{a-1}
\partial_tu+D[a(u)D^3u+a'(u)DuD^2u-Du^3]+u^2=0.
\end{equation}\end{remark}

In the following, the letters $C$, $C_i$ ($i = 1, 2,\cdots$) will
always denote positive constants different in various occurrences.

\section{Proof of Theorem \ref{thm1.1}}\label{sec2}





\begin{lemma}
\label{lem2.1}
If $u(x,t)$ is a solution of the problem (\ref{1-1})-(\ref{1-3}), then $u(x,t)$ satisfies
\begin{equation}
\label{2-6}
\|u\|_{L^{\infty}(Q_T)}\leq C,
~~
\sup_{t\in[0,T]}\|Du(\cdot,t)\|\leq C,
~~
\int_0^T\int_{\Omega}|D^3u(\cdot,t)|^2dxdt\leq C,
\end{equation}
\end{lemma}
\begin{proof}  Set $Nu=(-D^2)^{-1}u$ and $Nu_t=(-D^2)^{-1}u_t$. Multiplying both sides of the Eq.(\ref{1-1}) by $Nu$ and integrating over
$\Omega$, one gets
\begin{equation}
\frac12\frac d{dt}\|u\|_{H^{-1}(\Omega)}^2+\|u\|_{L^4(\Omega)}^4+\int_{\Omega}[a(u)+a'(u)u]|Du|^2dx=(u^2,Nu).
\label{2-8}\end{equation}
Note that
$$
(u^2,Nu)\leq\|u\|_{L^4(\Omega)}^2\|Nu\|\leq C\|u\|_{L^4(\Omega)}^2\|u\|\leq\frac14\|u\|_{L^4(\Omega)}^4+C\|u\|^2.
$$
 Summing up, we get
\begin{equation}
\label{2-9}
\frac12\frac d{dt}\|u\|_{H^{-1}(\Omega)}^2+\frac34\|u\|_{L^4(\Omega)}^4+\int_{\Omega}a(u)|D u|^2dx\leq C_1\|u\|^2.
\end{equation}
Multiplying both sides of the Eq.(\ref{1-1}) by $u$ and integrating over
$\Omega$, one gets
\begin{equation}
\label{2-10}
\frac12\frac d{dt}\|u\|^2+\int_{\Omega}a(u)|D^2 u|^2dx+3\int_{\Omega}u^2|D u|^2dx=-(u^2,u)\leq\frac14\|u\|_{L^4(\Omega)}^4+\|u\|^2.
\end{equation}
Combining (\ref{2-9}) and (\ref{2-10}) together gives
\begin{equation}
\label{2-11}\begin{aligned}&
\frac12\frac d{dt}(\|u\|_{H^{-1}(\Omega)}^2+\|u\|^2)+\int_{\Omega}a(u)(|D u|^2+|D^2 u|^2)dx+\frac12\|u\|_{L^4(\Omega)}^4\\&+3\|uD u\|^2
\leq C_2\|u\|^2
\leq C_2(\|u\|_{H^{-1}(\Omega)}^2+\|u\|^2),\end{aligned}
\end{equation}
which yields, owing to  Gronwall's inequality,
\begin{equation}
\label{2-12}
\|u\|_{H^{-1}(\Omega)}^2+\|u\|^2\leq e^{C_2T}(\|u_0\|_{H^{-1}(\Omega)}^2+\|u_0\|^2),\quad t\in[0,T].
\end{equation}
Integrating (\ref{2-11}) over $(0,T)$, we derive that
\begin{equation}
\label{2-13}
\int_0^T\int_{\Omega}a(u)(|D u|^2+|D^2u|^2)dxdt+\int_0^T\|u\|_{L^4(\Omega)}^4dt\leq C.
\end{equation}
Multiplying both sides of the Eq.(\ref{1-1}) by $Nu_t$ and integrating over $\Omega$, one gets
\begin{equation}\label{2-14}
\|u_t\|^2_{H^{-1}(\Omega)}-\int_{\Omega}a(u)u_tD^2 udx+\frac14\frac d{dt}\|u\|_{L^4(\Omega)}^4=(u^2,Nu_t),
\end{equation}
which means
\begin{equation}
\label{2-15}\begin{aligned}&
\|u_t\|_{H^{-1}(\Omega)}^2+\frac14\frac d{dt}\|u\|_{L^4(\Omega)}^4\\\leq& \varepsilon\|Nu_t\|^2+C_{\varepsilon}\|u\|_{L^4(\Omega)}^4+\int_{\Omega}a(u)|D^2 u|^2dx+\int_{\Omega}a(u)u_tdx.
\end{aligned}\end{equation}
Let $A(u)=\int_0^ua(s)ds$. Hence
\begin{equation}
\label{2-16}\begin{aligned}&
\|u_t\|_{H^{-1}(\Omega)}^2+\frac d{dt}\left(\frac14\|u\|_{L^4(\Omega)}^4-\int_{\Omega}A(u)dx\right)\\\leq& C_{\varepsilon}\|u\|_{L^4(\Omega)}^4+C_3\varepsilon\|u_t\|^2_{H^{-1}(\Omega)}+\int_{\Omega}a(u)|D^2 u|^2dx,
\end{aligned}
\end{equation}
where $\varepsilon$ is small enough, it satisfies $1-C_3\varepsilon>0$. Integrating (\ref{2-16})  over $(0,T)$, we deduce that
\begin{equation}
\label{2-17}\begin{aligned}&
\frac14\|u\|_{L^4(\Omega)}^4-\int_{\Omega}A(u)dx+C\int_0^T\|u_t\|_{H^{-1}(\Omega)}^2dt\\\leq& C_{\varepsilon}\int_0^T\|u\|_{L^4(\Omega)}^4dt+\int_0^T\int_{\Omega}a(u)|D^2 u|^2dx+\frac14\|u_0\|_{L^4(\Omega)}^4-\int_{\Omega}A(u_0)dx.\end{aligned}
\end{equation}
By (\ref{2-13}), we obtain
\begin{equation}\begin{aligned}
\label{2-18}\frac14\|u\|_{L^4(\Omega)}^4+C\int_0^T\|u_t\|_{H^{-1}(\Omega)}^2dt\leq &C+\int_{\Omega}\left[\int_0^ua(v)dv\right]dx\\\leq& C+\sup|a(u)|\int_{\Omega}|u|dx\leq C.\end{aligned}
\end{equation}
It then follows from (\ref{2-18}) that
\begin{equation}
\label{2-19}
\int_0^T\|u_t\|_{H^{-1}(\Omega)}^2dt=\int_0^T\|D(a(u)D^2u)-f'(u)D u+D (Ng(u))\|^2dt\leq C.
\end{equation}
Note that
\begin{equation}\begin{aligned}\label{2-20}
\|D(a(u)D^2 u)\|^2=&\|a(u)D^3 u+a'(u)D uD^2 u\|^2\\\leq& C(1+\|f'(u)D u\|^2+\|DNg(u)\|^2)
\\
\leq&C\left(1+3\int_{\Omega}u^4|Du|^2dx+\int_{\Omega}u^4dx\right).
\end{aligned}\end{equation}
Using Nirenberg's inequality, we derive that
\begin{equation}\label{2-21}
\|u\|_{L^8(\Omega)}\leq C(\|D^3 u\|^{\frac1{8}}\|u\|^{\frac78}+\|u\|),
\end{equation}
and
\begin{equation}\label{2-22}
\|Du\|_{L^4(\Omega)}\leq C(\|D^3u\|^{\frac5{12}}\|u\|^{\frac7{12}}+\|u\|).
\end{equation}
Combining (\ref{2-20})-(\ref{2-22}) together gives
\begin{equation}
\label{2-22-1}
\|D(a(u)D^2 u)\|^2\leq \varepsilon\|D^3u\|^2+C.
\end{equation}
Follows from (\ref{2-20}) and (\ref{2-22-1}), we obtain
\begin{equation}\label{2-22-2}\begin{aligned}
 \|a'(u)D uD^2u\|^2dt\leq & (\|D(a(u)D^2 u)\|^2+\|a(u)D^3 u\|^2)\\\leq& \varepsilon\|D^3u\|^2+C+\|a(u)D^3 u\|^2.\end{aligned}
\end{equation}
By (\ref{2-22-1}), we deduce that
\begin{equation}\label{2-23}\begin{aligned}
\int_0^T\|D(a(u)D^2 u)\|^2dt \leq& C+\varepsilon\int_0^T\|D^3 u\|^2dt.
\end{aligned}
\end{equation}
Thus
\begin{equation}
\begin{aligned}&
 \int_0^T\int_{\Omega}a'(u)|DuD^2 uD^3u|dxdt\\\leq&\frac1{M_1}\int_0^T\int_{\Omega}a(u)a'(u)|DuD^2 uD^3 u|dxdt
\\
\leq&\frac1{2M_1 }\left(\int_0^T\|a(u)D^3u\|^2dt+ \int_0^T\|a'(u)DuD^2u\|^2dt\right)
\\ \leq&
 \frac1{M_1 }\int_0^T\|a(u)D^3u\|^2dt+\frac{\varepsilon}{2M_1}\int_0^T\|D^3u\|^2dt+C
 \\ \leq&
 \left(\frac1{M_1 }+\frac{\varepsilon}{2M_1^3}\right)\int_0^T\|a(u)D^3u\|^2dt +C.
 \label{2-24}
\end{aligned}
\end{equation}
Multiplying both sides of  Eq.(\ref{1-1}) by $-D^2 u$ and integrating over $\Omega$, one gets
\begin{equation}\label{2-27}\begin{aligned}&
\frac12\frac d{dt}\|Du\|^2+\int_{\Omega}a(u)|D^3 u|^2dx\\=&-\int_{\Omega}a'(u)D uD^2 uD^3udx-\int_{\Omega}D^2 u^3Dudx-\int_{\Omega}2u|Du|^2dx.
\end{aligned}\end{equation}
By (\ref{2-21}), (\ref{2-22}), (\ref{2-24}) and (\ref{2-27}), we have
\begin{equation}\begin{aligned}&
\label{2-28}
\frac 12\frac d{dt}\|D u\|^2+\int_{\Omega}a(u)|D^3 u|^2dx\\\leq& \int_{\Omega}a'(u)|D uD^2 uD^3 u|dx+3\int_{\Omega}u^2D uD^3 udx-\int_{\Omega}2u|D u|^2dx
\\
\leq& \int_{\Omega}a'(u)|D uD^2 uD^3 u|dx+3\|D u\|_{L^4(\Omega)}\|u\|_{L^8(\Omega)}^2\|D^3 u\|+2\|u\|\|Du\|_{L^4(\Omega)}^2
\\
\leq&\left(\frac1{M_1 }+\frac{\varepsilon}{2M_1^3}\right)\int_0^T\|a(u)D^3u\|^2dt+\varepsilon\|D^3u\|^2+C\\
\leq&\left(\frac1{M_1 }+\frac{\varepsilon}{2M_1^3}+\varepsilon M_2^2\right)\int_0^T\|a(u)D^3u\|^2dt +C.
\end{aligned}
\end{equation}Note that $M_1>1$. Choosing $\varepsilon$ sufficiently small,
integrating (\ref{2-28}) over $(0,T)$, using (\ref{2-24}), we derive that
\begin{equation}
\label{2-29}
\frac d{dt}\|D u\|^2+\int_{\Omega}a(u)|D^3 u|^2dx\leq C.
\end{equation}
Hence
\begin{equation}
\label{2-30}
\|D u\|^2+\int_0^T\int_{\Omega}a(u)|D^3 u|^2dxdt\leq C.
\end{equation}

\end{proof}

\begin{lemma}
\label{lem2.5}Suppose that $\langle u_0\rangle=\int_{\Omega}u_0dx$ is bounded, $1<M_1\leq a(u)\leq M_2$, $a'(u)>0$ and $a'(u)u\geq0$, then for the solution $u(x,t)$ of problem (\ref{1-1})-(\ref{1-3}), we have
\begin{equation}
\label{2-31}
|u(x_1,t)-u(x_2,t)|\leq C|x_1-x_2|^{\frac12},\quad \forall t\in[0,T],x_1,x_2\in[0,1],
\end{equation}and
\begin{equation}
\label{2-32}
|u(x,t_1)-u(x,t_2)|\leq C|t_1-t_2|^{\frac18},\quad\forall x\in[0,1],t_1,t_2\in[0,T].
\end{equation}
\end{lemma}
\begin{proof}
The inequality (\ref{2-31}) can be obtained by the inequality (\ref{2-30}) directly.  Integrating Eq.(\ref{1-1}) with respect to $(x,t)$ over $(y,y+(\Delta t)^{\frac14})\times(t_1,t_2)$, we derive that
\begin{equation}
\begin{aligned}&
(\Delta t)^{\frac14}\int_{\Omega}(u(y+\theta(\Delta t)^{\frac14},t_2)-u(y+\theta(\Delta t)^{\frac14},t_1))d\theta
\\
=&-\int_{t_1}^{t_2}a(u(y+(\Delta t)^{\frac14},t)D^3u(y+(\Delta t)^{\frac14},t)dt+\int_{t_1}^{t_2}a(u(y,t))D^3u(y,t)dt
\\
&-\int_{t_1}^{t_2}a'(u(y+(\Delta t)^{\frac14},t))Du(y+(\Delta t)^{\frac14},t)D^2u(y+(\Delta t)^{\frac14},t)dt\\&-\int_{t_1}^{t_2}a(u(y,t))D[u(y,t)]^3dt
-\int_{t_1}^{t_2}DN[u(y+(\Delta t)^{\frac14},t)]^2dt\\&+\int_{t_1}^{t_2}DN[u(y,t)]^2dt+\int_{t_1}^{t_2}a'(u(u,t))Du(y,t)D^2u(y,t)dt\\&
+\int_{t_1}^{t_2}a(u(y+(\Delta t)^{\frac14},t))D[u(y+(\Delta t)^{\frac14},t)]^3dt
.
\end{aligned}\label{2-33}\end{equation}
Integrating (\ref{2-33}) with respect to $y$ over $(x,x+(\Delta t)^{\frac14})$, by H\"{o}lder's inequality and Sobolev's embedding theorem, we derive that
\begin{equation}
\begin{aligned}
&(\Delta t)^{\frac14}\int_x^{x+(\Delta t)^{\frac14}}\int_0^1[u(y+\theta(\Delta t)^{\frac14},t_2)-u(y+\theta(\Delta t)^{\frac14},t_1)]d\theta dy
\\
\leq&C\int_x^{x+(\Delta t)^{\frac14}}\int_0^1|Du(y+(\Delta t)^{\frac14},t)|dtdy+C\int_x^{x+(\Delta t)^{\frac14}}\int_0^1|Du(y,t)|dtdy
\\&+C\int_x^{x+(\Delta t)^{\frac14}}\int_0^1
|D^3u(y+(\Delta t)^{\frac14},t)|dtdy+C\int_x^{x+(\Delta t)^{\frac14}}\int_0^1|D^3u(y,t)|dtdy\\&+C\int_x^{x+(\Delta t)^{\frac14}}\int_0^1u(y+(\Delta t)^{\frac14},t)]dtdy+C\int_x^{x+(\Delta t)^{\frac14}}\int_0^1u(y,t)dtdy
\\
\leq&C(\Delta t)^{\frac58}\left[\left(\int_0^T\int_{\Omega}|Du(y,t)|^2dtdy\right)^{\frac12}+\left(\int_0^T\int_{\Omega}|D^3u(y,t)|^2dtdy\right)^{\frac12}\right.
\\&\left.+\left(\int_0^T\int_{\Omega}|u(y,t)|^2dtdy\right)^{\frac12}\right]
\leq C(\Delta t)^{\frac58}.
\end{aligned}\label{2-34}\end{equation}
By mean value theorem, we know that there is a point $x^*=y^*+\theta^*(\Delta t)^{\frac14}$, $y^*\in(x,x+(\Delta t)^{\frac14})$, $\theta^*\in(0,1)$, such that
$$
|u(x^*,t_1)-u(x^*,t_2)|\leq C|t_1-t_2|^{\frac18}.
$$
It then follows from (\ref{2-31}) and the above inequality that
\begin{equation}\begin{aligned}&
|u(x,t_1)-u(x,t_2)|\\\leq &|u(x,t_1)-u(x^*,t_1)|+|u(x^*,t_1)-u(x^*,t_2)|+|u(x,t_2)-u(x^*,t_2)|
\\
\leq& C|t_1-t_2|^{\frac18},
\end{aligned}\nonumber\end{equation}
holds for all $x\in[0,1]$ and $t_1,t_2\in[0,T]$. Then, the proof is complete.

\end{proof}

\begin{lemma}[See Liu\cite{Yin2}]Suppose that $\sup|f|<\infty$, $\tilde{a}(x,t)\in C^{\frac12,\frac18}(\bar{Q}_T)$, and there exist two constants $a_0$ and $A_0$ such that $0<a_0\leq \tilde{a}(x,t)\leq A_0$ for all $(x,t)\in Q_T:=(0,1)\times(0,T)$. If $u$ is a smooth solution of the following linear problem
\begin{equation}
\begin{aligned}
&\partial_tu+D[\tilde{a}(x,t)D^3u]=Df,\quad (x,t)\in Q_T,
\\
&u|_{x=0,1}=D^2u|_{x=0,1}=0,
\quad u(x,0)=0,
\end{aligned}\nonumber
\end{equation}
then, for any $\delta\in(0,\frac12)$, there is a constant $C$ depending on $a_0$, $A_0$, $\delta$, $T$ and $\|\tilde{a}\|_{C^{\frac12,\frac18}(\bar{Q}_T)}$, such that
$$
|u(x_1,t_2)-u(x_2,t_2)|\leq C(1+\sup|f|)(|x_1-x_2|^{\delta}+|t_1-t_2|^{\frac{\delta}4}).
$$
\label{lem2.6}

\end{lemma}

\begin{proof}
[Proof of Theorem \ref{thm1.1}]Based on Lemma \ref{lem2.6}, we obtain the H\"{o}lder estimate for $Du$:
$$
|Du(x_1,t_1)-Du(x_2,t_2)|\leq C(|x_1-x_2|^{\frac{\alpha}2}+|t_1-t_2|^{\frac{\alpha}8}).
$$
The conclusion follows immediately from the classical theory, since we can transform Eq.(\ref{1-1}) into the following form:
$$
\partial_tu+A_1(x,t)D^4u+A_2(x,t)D^3u+A_3(x,t)D^2u+A_4(x,t)Du+[u(x,t)]^2=0,
$$
where the H\"{o}lder norms on
\begin{equation}
\begin{aligned}&A_1(x,t)=a(u(x,t)),\quad\quad\quad\quad\quad\quad
A_2(x,t)=2a'(u(x,t))Du(x,t),
\\
&A_3(x,t)=a{''}(u(x,t))|Du(x,t)|^2,
\quad
A_4(x,t)=3[u(x,t)]^2,
\end{aligned}\nonumber\end{equation}
have been estimated in the above discussion. The proof is complete.

\end{proof}
\subsection*{Acknowledgment}
Dr. Zhao was supported by NSFC (grant
No. 11401258), NSF of Jiangsu Province (grant
No. BK20140130) and China Postdoctoral Science Foundation (grant No. 2015M581689
), Dr. Duan was supported by NSF of Jiangsu Province (grant No. 	BK20170172)  and China Postdoctoral Science Foundation (grant No. 2017M611684
).
Part of this work was done when Dr. Zhao was visiting the Institute of Mathematics for Industry of Kyushu University. He appreciate the hospitality of Prof. Fukumoto, MS. Sasaguri and IMI.

\end{document}